\documentclass[11pt,reqno]{amsart}
\usepackage[utf8]{inputenc} 
\usepackage[T1]{fontenc}    
\usepackage{hyperref}   
\usepackage{amsfonts}       
\usepackage{microtype}      

\usepackage{amscd,amssymb,amsmath,amsthm}
\usepackage[arrow,matrix]{xy}
\usepackage{graphicx}
\usepackage{multicol}
\usepackage{epstopdf}
\usepackage{color}
\newtheorem{thm}{Theorem}
\newtheorem{defn}{Definition}

\newtheorem{pro}{Proposition}
\newtheorem{rk}{Remark}

\numberwithin{equation}{section} 

\begin{document}

\vspace{0.5in}
\renewcommand{\bf}{\bfseries}
\renewcommand{\sc}{\scshape}
\vspace{0.5in}

\title[Leslie's prey-predator model in discrete time]{Leslie's prey-predator model in discrete time}

\author{U. A. Rozikov, S.K. Shoyimardonov}

 \address{U.\ A.\ Rozikov\\  V.I.Romanovskii Institute of mathematics,
81, Mirzo Ulug'bek str., 100125, Tashkent, Uzbekistan.}
\email {rozikovu@yandex.ru}

\address{Department of Higher Mathematics; Tashkent University of Information Technologies named after Muhammad al-Khwarizmi;
Tashkent, Uzbekistan.}
\email{shoyimardonov@inbox.ru}





\keywords{population, Leslie model, prey-predator, discrete-time.}

\subjclass[2010]{34D20 (92D25)}

\begin{abstract} We consider the Leslie's prey-predator model with discrete-time. This model is given by a non-linear
 evolution operator depending on five parameters. We show that this operator has two fixed points and define
 type of each fixed point depending on the parameters. Finding two invariant sets of the evolution operator
 we study the dynamical systems generated by the operator on each invariant set.
 Depending on the parameters we classify the dynamics between
a predator and a prey of the Leslie's model.
\end{abstract}

\maketitle

\section{Introduction}
First predator-prey models were introduced by Lotka and Volterra.
These models have been extensively studied by mathematical and biological researchers.
The investigations are important in understanding the dynamics between
a predator and a prey (population with two species), which live together in the same
environment (see \cite{Aziz}, \cite{Brit}, \cite{Ch}, \cite{M}, \cite{SB}, \cite{SS} and references therein). One interests for a suitable conditions that allow the both species survive in
equilibria \cite{RU}. But in many papers (see for example \cite{RGan}-\cite{Rob}) have shown that considering
a harvesting term in the model can lead to the extinction of any species.

Following \cite{Brit} consider the Leslie's prey-predator model in continuous time.
At time moment $t\geq 0$ consider the following model:

\begin{equation}\label{1}
\begin{cases}
\frac{dx}{dt}=ax-bx^2-cxy\\
\frac{dy}{dt}=dy-\alpha\frac{y^2}{x},
\end{cases}
\end{equation}
where $a,b,c,d$ and $\alpha$ are positive parameters. The predator equation is logistic, with carrying capacity proportional to the prey population. This two species food chain model describes a prey population $x$ which serves as food for a predator $y$.
This model usually studied for continuous time.

In this paper (as in \cite{RSH} and \cite{RV}) we study a model of discrete time process of Leslie's prey-predator model (\ref{1}), which has the following form
\begin{equation}\label{discr}
V:
\begin{cases}
x^{(1)}=x(a-1-bx-cy)\\
y^{(1)}=y(d-1-\alpha\frac{y}{x}).
\end{cases}
\end{equation}
where $(x,y)\in R^2_+=\{(u,v)\in R^2:u>0, v\geq0\}$ and $a>1,b>0, c>0, d>1, \alpha>0.$

We are interested to the behavior of the sequence $V^n(x,y)$, $n\geq 1$ for any initial point $(x,y)\in R^2_+$.

The paper is organized as follows. In Section 2 we construct two invariant sets with respect to operator $V$. Section 3 devoted to
fixed points of the operator, moreover the type of each fixed point is defined depending on the parameters of the model.
In Section 4 under some conditions on parameters we give limit points of trajectories.
In the last section we give numerical analysis of trajectories corresponding to the remaining cases of parameters.

\section{Invariant sets}

The set $M$ is called an \emph{invariant} with respect to operator $V$ if $V(M)\subset M.$ We formulate the following:
 \begin{pro} The set
 $$M_{1}=\left\{(x,y)\in R^2_+ \, : \, 0<x<\frac{a-1}{b}, \, y=0\right\}$$
 is an invariant with respect to $V$.
 \end{pro}
 The proof of this Proposition is straightforward.

 \begin{defn}\label{top} (see. \cite{De}, page 47) Let $f:A\rightarrow A$ and $g:B\rightarrow B$ be two maps. $f$ and $g$ are said to be topologically conjugate if there exists a homeomorphism $h:A\rightarrow B$ such that, $h\circ f=g\circ h$. The homeomorphism $h$ is called a topological conjugacy.
 \end{defn}
 Under a condition on parameters the restriction of the operator $V$ on $M_1$
 is topologically conjugate to the well-known quadratic
 family $F_\mu(x)=\mu x(1-x)$ (discussed in \cite{De}). Let us do this point clear:
 in the system  (\ref{discr}),  if $y=0$ then from the first equation we get $x(a-1-bx)=f_{a,b}(x).$
 \begin{pro} Two maps $F_\mu(x)$ and $f_{a,b}(x)$ are topologically conjugate for $\mu=3-a$.
 \end{pro}
 \begin{proof} We take the linear map $h(x)=px+q$ and by Definition \ref{top} we should have $h(F_\mu(x))=f_{a,b}(h(x)),$ i.e.,
 $$p\mu x(1-x)+q=(px+q)(a-1-b(px+q))$$
 from this identity we get
$$
\begin{cases}
p\mu=bp^2\\
p\mu=p(a-1-2bq)\\
q=q(a-1-bq)
\end{cases}
\Rightarrow
\begin{cases}
p=\frac{\mu}{b}\\
q=\frac{a-1-\mu}{2b}\\
q=\frac{a-2}{b}
\end{cases} \Rightarrow a=3-\mu .$$
Thus, the homeomorphism is $h(x)=\frac{3-a}{b}x+\frac{a-2}{b}.$ Moreover, since $a>1$ we have $\mu<2.$
\end{proof}
The importance of this proposition is that if two maps are topologically conjugate
then they have essentially the same dynamics (see. \cite{De}, page 53).

 \begin{pro} \label{inv} Let  $1<a\leq 2.$ Then the set
 $$M_{2}=\left\{(x,y)\in R^2_{+}: \, \frac{\alpha y}{d-1}\leq x<\frac{a-1-cy}{b}\right\}$$
   is an invariants with respect to operator $V$ if

   (1) $1<d\leq2$  or

   (2) $d<4a-3$ and $0<x<\frac{2\alpha\sqrt{bc\alpha(d-1)+b^2\alpha^2+c^2(a-1)(d-1)}-\alpha(c(d-1)+2b\alpha)}{c^2(d-1)}$.
   \end{pro}
   \begin{proof}
 If $(x,y)\in M_{2}$ then by $1<a\leq2$ we have
$$x^{(1)}=x(a-1-bx-cy)<x(a-1)\leq x,$$
and by the form of $M_2$ we have
$$x^{(1)}=x(a-1-bx-cy)> x(a-1-(a-1-cy)-cy)=0.$$
Similarly,
$$y^{(1)}=y(d-1-\alpha\frac{y}{x})\geq y(d-1-\alpha\frac{d-1}{\alpha})=0.$$
Next we show that $y^{(1)}\leq \frac{(d-1)x^{(1)}}{\alpha}.$

\emph{Case-1}. If $1<d\leq2$ then $y^{(1)}=y(d-1-\alpha\frac{y}{x})\leq y(d-1)\leq y$. Hence, in this case $x^{(1)}<x, \,\, y^{(1)}\leq y$ and $M_2$ is an invariant.

\emph{Case-2}. Let we consider the inequality $y^{(1)}\leq \frac{(d-1)x^{(1)}}{\alpha}.$ Instead of $x^{(1)}, y^{(1)}$ we put their expressions and we have
$$\frac{\alpha^2y^2}{x(d-1)}-(cx+\alpha)y+x(a-1-bx)\geq0.$$

Last inequality is always true with respect to $y$ if a discriminant is nonpositive, i.e.,
$$c^2x^2+(2c\alpha+\frac{4b\alpha^2}{d-1})x+\alpha^2(\frac{d-4a+3}{d-1})\leq0.$$
By solving this inequality we obtain the (2) condition of the Proposition.Thus the proposition is proved.
\end{proof}
\begin{rk} The conditions (to the parameters $a,b,c,d,\alpha$) in Proposition \ref{inv} are sufficient for the set $M_2$ to be an invariant.

\end{rk}
\section{Fixed points}

A \textbf{\emph{fixed point}} (\cite{RSH}) $p$ for a mapping $F: R^{m}\rightarrow R^{m}$ is a solution to the equation $F(p)=p$.
We will study fixed points of the operator (\ref{discr}). Let $\lambda^{(0)}=(x^{(0)}, y^{(0)})$ be an initial point.
By the continuity of the operator $V$, (\ref{discr}), the limit points of each trajectory $\lambda^{(n)}=V^n(\lambda^{(0)})$
 are fixed points for the operator $V$.

\begin{pro} For the operator (\ref{discr}) fixed points are
$$(i)  \ \ \ \ \ \   \lambda_{1}=\left(\frac{a-2}{b},0\right),  (a>2) $$ and
$$(ii) \ \  \lambda_{2}=\left(\frac{(a-2)\alpha}{b\alpha+c(d-2)}, \frac{(a-2)(d-2)}{b\alpha+c(d-2)}\right)$$
where $a>2, b\alpha+c(d-2)>0$ (if $1<a<2, b\alpha+c(d-2)<0$ then from second inequality we get that $d<2$ and $\frac{(a-2)(d-2)}{b\alpha+c(d-2)}<0$.)
\end{pro}

\begin{proof}
i) If $y=0$ then the equation $x=x(a-1-bx)$, (where $x>0$) has the solution $x=\frac{a-2}{b}$.\\
ii) If $y>0$ then the system of equations $x=x(a-1-bx-cy), y=y(d-1-\alpha\frac{y}{x})$ has unique solution   $\left(\frac{\alpha(a-2)}{b\alpha+c(d-2)},\frac{(a-2)(d-2)}{b\alpha+c(d-2)}\right)$ and here for
existence of fixed point we have conditions $a>2$ and $b\alpha+c(d-2)>0$. If $a=2$ then
\begin{equation}\label{2}
\begin{cases}
x^{(1)}=x(1-bx-cy)=x\\
y^{(1)}=y(d-1-\alpha\frac{y}{x})=y
\end{cases}
\end{equation}
from this we get $bx=-cy$, but all parameters are positive, so $x=y=0$ and there is no solution of the system (\ref{2}).
\end{proof}

\begin{pro}\label{pr} The following relations hold
\begin{itemize}
\item[(1)]
$$\lambda_{1}=\left\{\begin{array}{ll}
{\rm nonhyperbolic}, \ \ {\rm if} \ \ a=4 \ \ or \ \ d=2  \\[2mm]
{\rm attractive}, \ \ \ \ {\rm if} \ \ 2<a<4, 1<d<2 \\[2mm]
{\rm repeller}, \ \ \ \ {\rm if} \ \  a>4, d>2 \\[2mm]
{\rm saddle}, \ \ \  \ \ \ {\rm if}  \ \ {\rm otherwise} \ \
\end{array}\right.$$
\end{itemize}

\begin{itemize}
\item[(2)]
$$\lambda_{2}=\left\{\begin{array}{lll}
{\rm nonhyperbolic}, \ \ {\rm if} \ \  |\mu_{1}|=1, \ \ {\rm or} \ \ |\mu_{2}|=1 \\[2mm]
{\rm attractive}, \ \ \ \ {\rm if} \ \ |\mu_{1}|<1,|\mu_{2}|<1 \\[2mm]
{\rm repeller}, \ \ \ \ {\rm if} \ \ |\mu_{1}|>1,|\mu_{2}|>1\\[2mm]
{\rm saddle}, \ \ \  \ \ \ {\rm if}  \ \ {\rm otherwise} \ \
\end{array}\right.$$
\end{itemize}
where
$$\mu_{1}=-\frac{1}{2}\frac{1}{b\alpha+cd-2c}(ab\alpha+bd\alpha+cd^2-6b\alpha-6cd+8c+\sqrt{D}),$$
$$\mu_{2}=-\frac{1}{2}\frac{1}{b\alpha+cd-2c}(ab\alpha+bd\alpha+cd^2-6b\alpha-6cd+8c-\sqrt{D}).$$
$$D=a^2b^2\alpha^2-2ab^2d\alpha^2-6abcd^2\alpha-4ac^2d^3+b^2d^2\alpha^2+2bcd^3\alpha+c^2d^4+24abcd\alpha$$
$$+24ac^2d^2-24abc\alpha-48ac^2d-24bcd\alpha-24c^2d^2+32ac^2+32bc\alpha+64c^2d-48c^2.$$
\end{pro}

\begin{proof}
(1) First we find the Jacobian for the system (\ref{discr}):
\begin{equation}\label{jac}
\textbf{J}=\begin{bmatrix}
a-1-2bx-cy & -cx\\
\alpha\frac{y^2}{x^2} & d-1-2\alpha\frac{y}{x}
\end{bmatrix}
\end{equation}
Then the Jacobian at the fixed point $\lambda_{1}=\left(\frac{a-2}{b},0\right)$ has the form
\begin{center}
$\textbf{J}(\lambda_{1})=\begin{bmatrix}
3-a & -c\frac{(a-2)}{b} \\
0 & d-1
\end{bmatrix}$
\end{center}
and the eigenvalues of this matrix are $\nu_{1}=3-a, \nu_{2}=d-1$.
By solving inequalities  $|\nu_{1}|<1$, $|\nu_{2}|<1$ we get $2<a<4, 1<d<2$,
thus,  $\lambda_{1}$ is an attracting fixed point.
The proof of all other cases are similar.

(2) The Jacobian at the second fixed point
$\lambda_{2}=\left(\frac{(a-2)\alpha}{b\alpha+c(d-2)}, \frac{(a-2)(d-2)}{b\alpha+c(d-2)}\right)$ has the following form
 \begin{center}
$\textbf{J}(\lambda_{2})=\begin{bmatrix}
1-\frac{(a-2)b\alpha}{b\alpha+c(d-2)} & -\frac{(a-2)c\alpha}{b\alpha+c(d-2)} \\
\frac{(d-2)^2}{\alpha} & 3-d
\end{bmatrix}$.
\end{center}
Then the eigenvalues $\mu_{1}, \mu_{2}$ of this matrix are the roots of the following quadratical equation:
$$\mu^2-\mu\left(4-d-\frac{(a-2)b\alpha}{b\alpha+c(d-2)}\right)+(3-d)\left(1-\frac{(a-2)b\alpha}{b\alpha+c(d-2)}\right)+\frac{c(a-2)(d-2)^2}{b\alpha+c(d-2)}=0.$$
\end{proof}

We note that there exist coefficients satisfying the condition    $|\mu_{1}|<1,|\mu_{2}|<1.$ For example, if $a=3, b=1, c=2, d=4.5, \alpha=0.5$ then $\mu_{1}=0.3114, \mu_{2}=-0,9023.$

\section{Limit points}

\subsection{Definitions} The set of limit points of trajectory is very important in the theory of dynamical systems, so we will study the
set of  limit points of trajectories of the operator (\ref{discr}).
\begin{defn} \label{toptr} (see. \cite{De}, page 49) $f:J\rightarrow J$ is said to be topologically transitive if for any pair of open sets $U,V\subset J$ there exists $k>0$ such that $f^{k}(U)\cap V\neq\emptyset.$
\end{defn}
\begin{defn} \label{sens} (see. \cite{De}, page 49) $f:J\rightarrow J$ has sensitive dependence on initial conditions if there exists $\delta>0$ such that, for any $x\in J$ and any neighborhood $N$ of $x$, there exists $y\in N$ and $n\geq0$ such that $|f^{n}(x)-f^{n}(y)|>\delta.$
\end{defn}
\begin{defn} \label{chaos} (see. \cite{De}, page 50)  $f:J\rightarrow J$ is said to be \textbf{chaotic} on $J$ if

  1. $f$ has sensitive dependence on initial conditions;

  2. $f$ is topologically transitive;

  3. periodic points are dense in $J.$
  \end{defn}

\subsection{On the invariant set $M_1$} Consider trajectories on the invariant set $M_1$ first.
\begin{pro}\label{zero} If  $1<a\le2$ and initial point $(x^{(0)}, y^{(0)})\in M_1$ then $$\lim_{n\to \infty} (x^{(n)}, y^{(n)})=(0,0).$$
\end{pro}
\begin{proof} If $1<a<2$ then $x^{(1)}=x^{(0)}(a-1-bx^{(0)})<x^{(0)}(a-1)$, from this we get $x^{(n+1)}<x^{(0)}(a-1)^n.$ Thus, $$\lim_{n\to \infty} x^{(n)}=0$$  since $a-1<1.$ In addition, by conjugacy of the operators $F_\mu(x)$, $f_{a,b}(x)$ and by Proposition 5.3. (in \cite{De}, page 32) from $1<\mu<3$ we have $0<a<2.$ In Proposition \ref{top}, $f_{a,b}=h\circ F_{\mu}\circ h^{-1}$ since $h\circ F_{\mu}=f_{a,b}\circ h$, so $f_{a,b}^{n}=h\circ F_{\mu}^{n}\circ h^{-1}$ and going to the limit from two sides we have  $\mathop{\lim }\limits_{n\to \infty} f_{a,b}^{n}(x)=\mathop{\lim }\limits_{n\to \infty} h\circ F_{\mu}^{n}\circ h^{-1}(x)=h(\frac{1-\mu}{\mu})=0$,  where $h(x)=\frac{3-a}{b}x+\frac{a-2}{b}.$
If $a=2$ then $x^{(n+1)}=x^{(n)}(1-bx^{(n)})<x^{(n)}$ for any $n\in N$ (where $N$ is the set of all positive integers),
i.e., the sequence  $x^{(n)}$ monotonically decreasing. Since the sequence is bounded from below, it has a limit.
 The limit should be a fixed point for the function $f_{2,b}(x)$, i.e. the unique fixed point $x=0$. Hence, $x^{(n)}\rightarrow0$ as $n\rightarrow\infty.$

We note that, by the domain of the operator (\ref{discr}),  $x^{(1)}=x^{(0)}(a-1-bx^{(0)})>0 \Rightarrow x^{(0)}<\frac{a-1}{b}.$
\end{proof}

\begin{pro}\label{pro1} Let $2<a<4$.

(i)  $f_{a,b}(x)=x(a-1-bx)$ has an attracting fixed point $p_0=\frac{a-2}{b}$ and repelling fixed point 0.

(ii)  If $0<x<\frac{a-1}{b}$ then $$\lim_{n\to \infty} f_{a,b}^{n}(x)=p_0$$
\end{pro}
\begin{proof} (i). Let $f_{a,b}(x)=x(a-1-bx)$ with $a>2$.
Then it has two fixed points: $0$ and $p_0=\frac{a-2}{b}.$
We have $f'_{a,b}(0)=a-1$ and  $f'_{a,b}(p_0)=3-a$.
Hence 0 is a repelling for $a>2$ and $p_0$ is an
attracting with $2<a<4$.

(ii)  \emph{Case:} $2<a<3.$ Suppose $x\in (0, \frac{a-1}{2b}).$ Then graphical analysis
(which is called\emph{ Kyonigsa-Lamereya diagram}, see \cite{Shar}, page 7) shows
that $\lim_{n\to \infty} f_{a,b}^{n}(x)=p_0.$  If $x$ lies in the interval
$(\frac{a-1}{2b}, \frac{a-1}{b})$ then  $f_{a,b}(x)$ lies in $(0, \frac{a-1}{2b})$,
so that the previous argument implies (see Fig.\ref{f1})
$$f_{a,b}^{n}(x)=f_{a,b}^{n-1}(f_{a,b}(x))\rightarrow p_0, \ \ {\rm as} \ \ n\to\infty.$$
 \emph{Case:} $3<a<4.$ Graphical analysis shows what is different
 in this case (see Fig.\ref{f2}). Note that $\frac{a-1}{2b}<p_0<\frac{a-1}{b}.$
 Let $\hat{p}_0$ denote the unique point in the interval $(0,\frac{a-1}{2b})$
 that is mapped onto $p_0$ by $f_{a,b}.$ Then we can easily check that $f_{a,b}^2$
 maps the interval $[\hat{p}_0,p_0]$ inside $[\frac{a-1}{2b}, p_0].$
 It follows that $f_{a,b}^{n}(x)\rightarrow p_0$ as $n\rightarrow\infty$
 for all $x\in[\hat{p}_0,p_0].$ Now suppose $x<\hat{p}_0.$
 Again graphical analysis shows that there exists integer $k>0$
 such that $f_{a,b}^{k}(x)\in[\hat{p}_0,p_0].$
 Thus $f_{a,b}^{k+n}(x)\rightarrow p_0$ as $n\rightarrow\infty.$
 Similarly, $f_{a,b}$ maps the interval $(p_0, \frac{a-1}{b})$ onto $(0, p_0).$
 Since $(0,\frac{a-1}{b})=(0, \hat{p}_0)\cup[\hat{p}_0,p_0]\cup(p_0, \frac{a-1}{b})$,
 we have finished the proof.
\end{proof}

If $a>4$ then the fixed point $p_0$ becomes repelling.
Let us consider 2-periodical points of the function $f_{a,b}$ (see Fig.\ref{f3}) as roots of the equation:
$$\frac{f_{a,b}(f_{a,b}(x))-x}{f_{a,b}(x)-x}=0.$$
Then we have the following solutions:
$p_1=\frac{a-\sqrt{a(a-4)}}{2b}$, $p_2=\frac{a+\sqrt{a(a-4)}}{2b}.$
 We know that if $|f_{a,b}'(p_1)\cdot f_{a,b}'(p_2)|<1$ then the cycle $\{p_1,p_2\}$ is an attracting (see \cite{Shar}, page 9). Hence,
 $$f_{a,b}'(p_1)\cdot f_{a,b}'(p_2)=(a-1)^2-2b(a-1)(p_1+p_2)+4b^2p_1p_2=-a^2+4a+1,$$
then $|f_{a,b}'(p_1)\cdot f_{a,b}'(p_2)|<1$, if $4<a<2+\sqrt{6}\approx4.4494...$

If $a>2+\sqrt{6}$ then the cycle $\{p_1,p_2\}$ becomes repelling.
Numerical analysis shows that if $2+\sqrt{6}<a<4.543$  then there exists four periodical attracting cycle (see Fig.\ref{f4}).

\begin{pro}\label{oned} Let $x<\frac{a-1}{b}.$ Then

(i) If $4<a<2+\sqrt{6}\approx4.4494$ then the operator $f_{a,b}(x)=x(a-1-bx)$  has a cycle of period two;

(ii) If $2+\sqrt{6}<a<4.543$ then the operator $f_{a,b}(x)=x(a-1-bx)$  has a cycle of period four;

(iii) If $4.544<a<4.564$ then the operator $f_{a,b}(x)=x(a-1-bx)$  has a cycle of period eight;

(iv) If $a>3+\sqrt{5}$ then $f_{a,b}(x)=x(a-1-bx)$ has chaotic dynamics for any initial point $x^{0}\in (0,\frac{a-1}{b})\setminus\{p_0\}.$
\end{pro}

\begin{proof} The proof of (i), (ii) and (iii) follows from above mentioned discussion. Proof of the (iv) follows by the arguments of \cite{De}, pages 50-51.
\end{proof}
\begin{figure}[h!]
\begin{multicols}{2}
\hfill
\includegraphics[width=5cm]{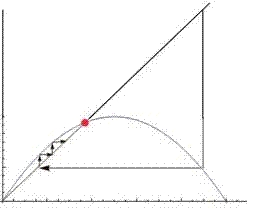}
\hfill
\caption{Graphical analysis of $f_{a,b}(x)=x(a-1-bx)$ when $2<a<3.$}\label{f1}
\label{figLeft}
\hfill
\includegraphics[width=5.6cm]{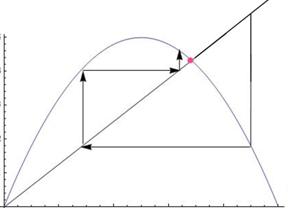}
\hfill
\caption{Graphical analysis of $f_{a,b}(x)=x(a-1-bx)$ when $3<a<4.$}\label{f2}
\label{figRight}
\end{multicols}
\end{figure}

\begin{figure}[h!]
\begin{multicols}{2}
\hfill
\includegraphics[width=6cm]{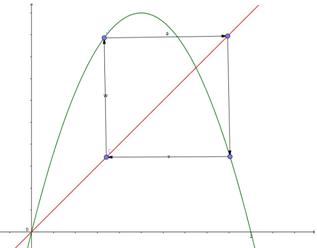}
\hfill
\caption{ Two periodical graphics of $f_{a,b}(x)=x(a-1-bx)$ when $a>4.$}\label{f3}
\label{figLeft}
\hfill
\includegraphics[width=6cm]{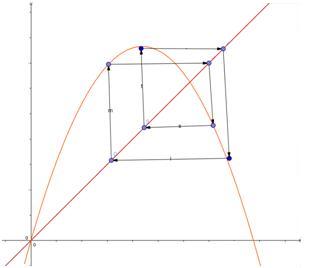}
\hfill
\caption{Four periodical graphics of $f_{a,b}(x)=x(a-1-bx)$ when $a>4.$}\label{f4}
\label{figRight}
\end{multicols}
\end{figure}
\subsection{On the invariant set $M_2$}
\begin{thm} Let $(x^{(0)}, y^{(0)})\in M_{2}$ be an initial point and $1<d\leq2.$

(i) If $1<a\leq 2 $ then

$$\lim_{n\to \infty} (x^{(n)}, y^{(n)})=(0,0)$$

(ii)  If $2<a<4$ then there exists a neighborhood $U(\lambda_1)$
such that
$$\lim_{n\to \infty} (x^{(n)}, y^{(n)})=\lambda_1=(\frac{a-2}{b},0), \ \ \forall (x^{(0)}, y^{(0)})\in U(\lambda_1).$$
\end{thm}

\begin{proof} For  $1<d<2$ we get  $y^{(n+1)}< y^{(0)}(d-1)^n$ and from this
$$\lim_{n\to \infty} y^{(n)}=0$$
If  $d=2$ then $y^{(n+1)}= y^{(n)}(1-\alpha\frac{y^{(n)}}{x^{(n)}})< y^{(n)}$. Thus, $y^{(n)}$ is decreasing and it has a limit.
We assume that
$$\mathop{\lim }\limits_{n\to \infty} y^{(n)}=\bar{y}\neq0.$$
From second equation of the operator $V$ we have:
$$x^{(n)}=\frac{\alpha y^{(n)}}{1-\frac{y^{(n+1)}}{y^{(n)}}}.$$
From this we obtain  $\mathop{\lim }\limits_{n\to \infty} x^{(n)}=\infty.$ This is a contradiction to
the boundedness of the sequence $x^{(n)}$ in the invariant set $M_2.$
Hence, $$\mathop{\lim }\limits_{n\to \infty} y^{(n)}=0.$$

\emph{Case-(i).} If $1<a<2$ we get  $x^{(n+1)}<x^{(0)}(a-1)^n$ and
$\lim_{n\to \infty} x^{(n)}=0.$

If $a=2$ then $x^{(1)}=x^{(0)}(1-bx^{(0)}-cy^{(0)}).$ It means than the sequence $x^{(n)}$
monotone decreasing and it has limit $\bar{x}.$ If we assume that  $\bar{x}\neq0$ then it
must be a fixed point. But in this case there is no fixed point of the operator $V.$ Thus, $\bar{x}=0.$

\emph{Case-(ii).} Above we have shown that independently on $a\in(1,4)$ the sequence $y^{(n)}$ has zero limit and $x^{(n)}\rightarrow0$ if $1<a\leq2.$

Let now $2<a<4.$
First equation of the operator $V$ is:
\begin{equation}\label{ss}
x^{(n+1)}=x^{(n)}(a-1-bx^{(n)}-cy^{(n)})
\end{equation}
For $y^{(0)}=0$ by Proposition \ref{pro1} we get that the limit of the sequence $x^{(n)}$ is $\frac{a-2}{b}$.

If $y^{(0)}>0$ then by Proposition \ref{pr} the fixed point $\lambda_1$ is attractive, and
therefore by the general theory of dynamical systems \cite{De},
there exists its neighborhood $U(\lambda_1)$
such that for any $(x^{(0)}, y^{(0)})\in U(\lambda_1)$ we have
$$\lim_{n\to \infty} (x^{(n)}, y^{(n)})=\lambda_1.$$
(see Fig.\ref{f5} for an illustration).
The theorem is proved.
\end{proof}

\begin{figure}[h!]
\includegraphics[width=10cm]{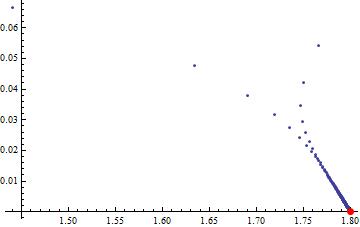}
\caption{a=3.8, b=1, c=2, d=2, $\alpha=4, x^{(0)}=1.2, y^{(0)}=0.2$. Shown $\lambda^{(n)}=(x^{(n)}, y^{(n)})$,  $n=0,1,\dots, 10000$ }\label{f5}
\end{figure}

\begin{thm} Let $(x^{(0)}, y^{(0)})\in M_{2}$ be an initial point which is not fixed points and let $1<d\leq2$. Then

(i) If $4<a<2+\sqrt{6}\approx4.4494$ then the operator $V^n(x^{(0)}, y^{(0)})$  converges to a cycle of period two;

(ii) If $2+\sqrt{6}<a<4.543$ then the operator $V^n(x^{(0)}, y^{(0)})$  converges to a cycle of period four;

(iii) If $4.544<a<4.564$ then the operator $V^n(x^{(0)}, y^{(0)})$  converges to a cycle of period eight;

(iv) If $a>3+\sqrt{5}$ then $V$ has chaotic dynamics.
\end{thm}
\begin{proof}By the condition $1<d\leq2$ we have that $\mathop{\lim }\limits_{n\to \infty} y^{(n)}=0.$ It means that for any initial point there exists $n_0\in N$ such that the operators $V$ and $f_{a,b}$ have the same limit behavior. Hence, proof of this theorem follows from Proposition \ref{oned}.\end{proof}

Let us give some figures related to this Theorem:
If $4<a<2+\sqrt{6}$ then by two periodical points $p_1=\frac{a-\sqrt{a(a-4)}}{2b}$, $p_2=\frac{a+\sqrt{a(a-4)}}{2b}$ which are mentioned in above, we have two attracting fixed points $(p_1,0), (p_2,0)$ of period 2. For example, if $a=4.3, b=1$ then $(p_1,0)\approx(1.5821,0)$ and $(p_2,0)\approx(2.71789,0)$ (Fig. \ref{f6}). Similarly, for the $2+\sqrt{6}<a<4.543$, for example, $a=4.5$, we can see the behavior of the trajectory with respect to the attracting cycle of period four (Fig.\ref{f7}). In addition, represented 8, 16 and greater periodical fixed points. (Fig. \ref{8-1}, Fig. \ref{16-1}, Fig. \ref{f8}, Fig. \ref{f9}).

\begin{figure}[h!]
\begin{multicols}{2}
\hfill
\includegraphics[width=7cm]{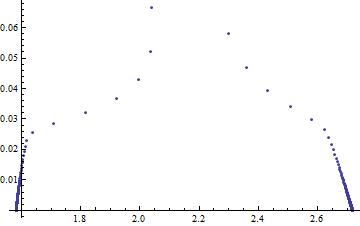}
\hfill
\caption{Two periodical: a=4.3, b=1, c=2, d=2, $\alpha=4, x^{(0)}=1.2, y^{(0)}=0.2$. Shown $\lambda^{(n)}$,  $n=0,1,\dots, 100000.$}\label{f6}
\label{figLeft}
\hfill
\includegraphics[width=7cm]{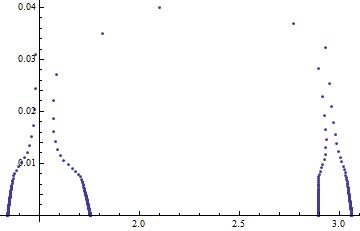}
\hfill
\caption{Four periodical: a=4.5, b=1, c=2, d=2, $\alpha=4, x^{(0)}=1.2, y^{(0)}=0.2$. Shown $\lambda^{(n)}$,  $n=0,1,\dots, 100000.$}\label{f7}
\label{figRight}
\end{multicols}
\end{figure}

\begin{figure}
\begin{multicols}{2}
\hfill
\includegraphics[width=7cm]{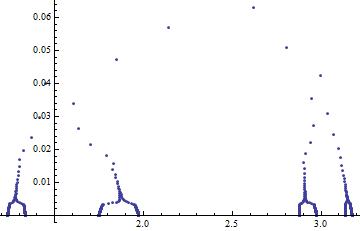}
\hfill
\caption{Eight periodical: a=4.564, b=1, c=2, d=2, $\alpha=4, x^{(0)}=1.2, y^{(0)}=0.2$. Shown $\lambda^{(n)}$,  $n=0,1,\dots, 100000.$}\label{8-1}
\label{figLeft}
\hfill
\includegraphics[width=7cm]{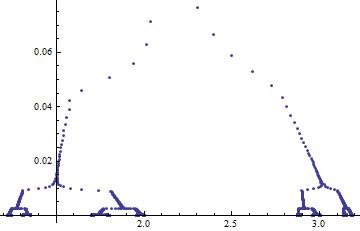}
\hfill
\caption{Sixteen periodical: a=4.569, b=1, c=5, d=2, $\alpha=42, x^{(0)}=1.2, y^{(0)}=0.2$. Shown $\lambda^{(n)}$,  $n=0,1, \dots, 100000.$}\label{16-1}
\label{figRight}
\end{multicols}
\end{figure}

\begin{figure}[h!]
\begin{multicols}{2}
\hfill
\includegraphics[width=7cm]{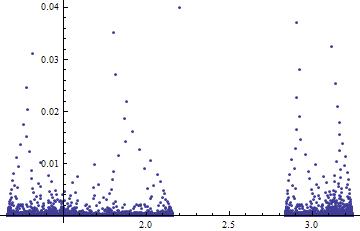}
\hfill
\caption{a=4.6, b=1, c=2, d=2, $\alpha=4, x^{(0)}=1, y^{(0)}=0.2.$ Shown $\lambda^{(n)}$,  $n=0,1,\dots, 100000.$}\label{f8}
\label{figLeft}
\hfill
\includegraphics[width=7cm]{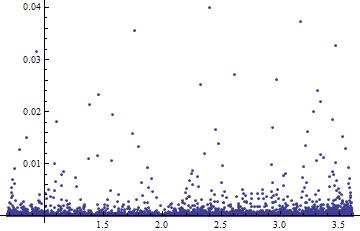}
\hfill
\caption{a=4.8, b=1, c=2, d=2, $\alpha=4, x^{(0)}=1, y^{(0)}=0.2.$  Shown $\lambda^{(n)}$,  $n=0,1,\dots, 100000.$}\label{f9}
\label{figRight}
\end{multicols}
\end{figure}

\section{Case $d>2, a>2$}
In this case we have several interesting cases, in particular chaos.
Numerical analysis shows that the coordinates of the vector $\lambda^{(n)}$ are not monotone, so it is not easy to see the limit properties of the trajectory.
Therefore, we study these limits numerically for concrete values of parameters:
(In all Figures the red point is the fixed point $\lambda_2$)

\subsection{Numerical analysis.}

1) $a=3, b=2, c=5, d=4, \alpha=1.$ Then by the system (\ref{discr}) we get
\begin{equation}\label{d1}
\begin{cases}
x^{(n+1)}=2x^{(n)}-2(x^{(n)})^2-5x^{(n)}y^{(n)}\\[2mm]
y^{(n+1)}=3y^{(n)}-\frac{(y^{(n)})^2}{x^{(n)}}
\end{cases}
\end{equation}
For this system $\lambda_2=(\frac{1}{12},\frac{2}{12})=(0.0833, 0.1666).$
Here we choose the initial point with condition $x^{(1)}>0, y^{(1)}\geq0$.
By using Wolfram Mathematica 7.0 we find limit points of initial
point $x^{(0)}=0.1, y^{(0)}=0.2$ (Fig.\ref{f10}).

\begin{figure}[h!]
\begin{multicols}{2}
\hfill
\includegraphics[width=7cm]{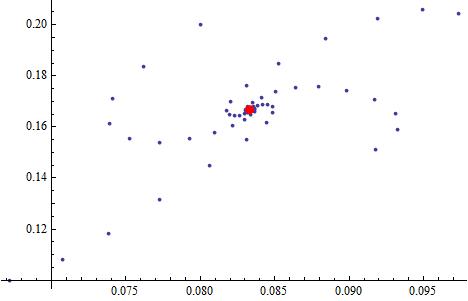}
\hfill
\caption{a=3, b=2, c=5, d=4, $\alpha=1, x^{(0)}=0.1, y^{(0)}=0.2.$  Shown $\lambda^{(n)}$,  $n=0,1,\dots, 2000.$}\label{f10}
\label{figLeft}
\hfill
\includegraphics[width=7cm]{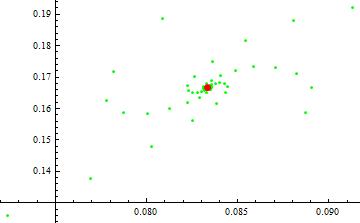}
\hfill
\caption{a=3, b=2, c=5, d=4, $\alpha=1, x^{(0)}=0.09514, y^{(0)}=0.1919.$  Shown $\lambda^{(n)}$,  $n=0,1,\dots, 2000.$}\label{f10-1}
\label{figRight}
\end{multicols}
\end{figure}

\begin{figure}[h!]
\begin{multicols}{2}
\hfill
\includegraphics[width=7cm]{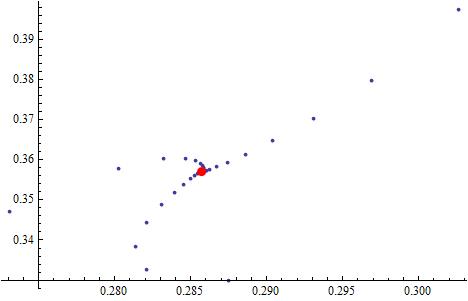}
\hfill
\caption{a=3, b=1, c=2, d=4.5, $\alpha=2, x^{(0)}=0.25, y^{(0)}=0.3.$  Shown $\lambda^{(n)}$,  $n=0,1,\dots, 2000.$}\label{f11}
\label{figLeft}
\hfill
\includegraphics[width=7cm]{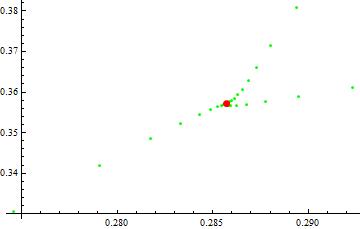}
\hfill
\caption{a=3, b=1, c=2, d=4.5, $\alpha=2, x^{(0)}=0.2967, y^{(0)}=0.364.$  Shown $\lambda^{(n)}$,  $n=0,1,\dots, 2000.$}\label{f11-1}
\label{figRight}
\end{multicols}
\end{figure}

\begin{figure}[h!]
\begin{multicols}{2}
\hfill
\includegraphics[width=7cm]{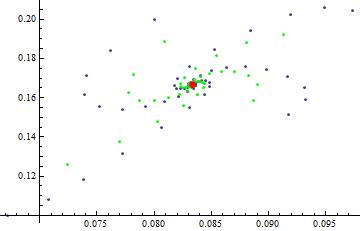}
\hfill
\caption{Fig.\ref{f10} and Fig.\ref{f10-1} in one system of coordinates}\label{c1}
\label{figLeft}
\hfill
\includegraphics[width=7cm]{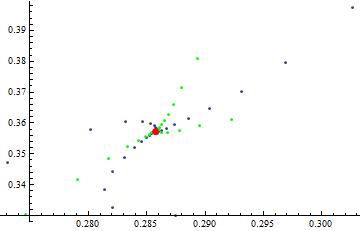}
\hfill
\caption{Fig.\ref{f11} and Fig.\ref{f11-1} in one system of coordinates}\label{c2}
\label{figRight}
\end{multicols}
\end{figure}

2) $a=3, b=1, c=2, d=4.5, \alpha=2.$ Then $\lambda_2=(\frac{2}{7},\frac{5}{14})=(0.2857, 0.3571)$ after $n=20000$ iteration we have $(x^{(n)},y^{(n)})=({0.285714, 0.357143})$  with $x^{(0)}=0.25, y^{(0)}=0.3$. (Fig.\ref{f11})
We see that in some subcases of the system (\ref{discr}) limit is:
$$\lim_{n\to \infty} (x^{(n)}, y^{(n)})=\lambda_2=\left(\frac{(a-2)\alpha}{b\alpha+c(d-2)}, \frac{(a-2)(d-2)}{b\alpha+c(d-2)}\right).$$
But for some cases the behavior of the  trajectory is various (Fig.\ref{f121}- Fig.\ref{f132}).

3) For $a=3.7, b=2, c=1, d=3.9, \alpha=3$ the trajectory is given in Fig.\ref{f121} and Fig.\ref{f122}

4) For $a=3.7, b=2, c=2, d=3.6, \alpha=3$ the trajectory is given in Fig.\ref{f131} and Fig.\ref{f132}
From the Fig.\ref{f131} and Fig.\ref{f132} we can say that in this case there is an invariant domain which is bounded by closed and attracting invariant curve.
\begin{figure}[h!]
\begin{multicols}{2}
\hfill
\includegraphics[width=7cm]{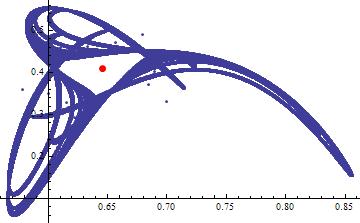}
\hfill
\caption{a=3.7, b=2, c=1, d=3.9, $\alpha=3, x^{(0)}=0.5, y^{(0)}=0.3.$ Shown $\lambda^{(n)}$,  $n=0,1,\dots, 100000.$}\label{f121}
\label{figLeft}
\hfill
\includegraphics[width=7cm]{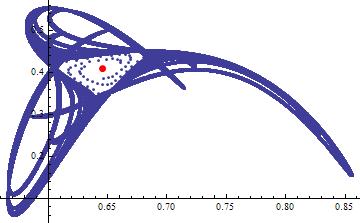}
\hfill
\caption{a=3.7, b=2, c=1, d=3.6, $\alpha=3, x^{(0)}=0.6431, y^{(0)}=0.3857.$ Shown $\lambda^{(n)}$,  $n=0,1,\dots, 100000.$}\label{f122}
\label{figRight}
\end{multicols}
\end{figure}

\begin{figure}[h!]
\begin{multicols}{2}
\hfill
\includegraphics[width=7cm]{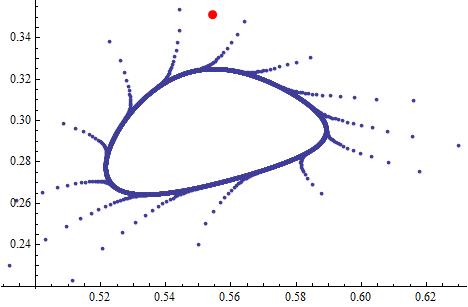}
\hfill
\caption{a=3.7, b=2, c=2, d=3.6, $\alpha=3, x^{(0)}=0.5, y^{(0)}=0.3.$ Shown $\lambda^{(n)}$,  $n=0,1,\dots, 100000.$}\label{f131}
\label{figLeft}
\hfill
\includegraphics[width=7cm]{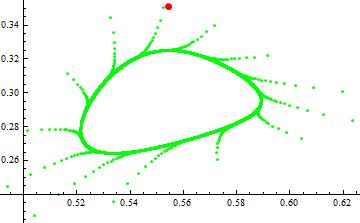}
\hfill
\caption{a=3.7, b=2, c=2, d=3.6, $\alpha=3, x^{(0)}=0.512, y^{(0)}=0.3168.$ Shown $\lambda^{(n)}$,  $n=0,1,\dots, 100000.$}\label{f131-1}
\label{figRight}
\end{multicols}
\end{figure}

\begin{figure}[h!]
\includegraphics[width=7cm]{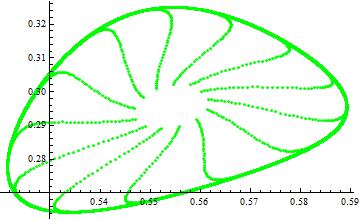}
\caption{a=3.7, b=2, c=2, d=3.6, $\alpha=3, x^{(0)}=0.5559, y^{(0)}=0.2901.$ Shown $\lambda^{(n)}$,  $n=0,1,\dots,100000.$}\label{f132}
\end{figure}

\begin{figure}[h!]
\begin{multicols}{2}
\hfill
\includegraphics[width=7cm]{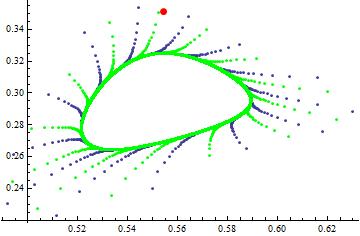}
\hfill
\caption{Fig.\ref{f131} and Fig.\ref{f131-1} in one system of coordinates}\label{c3}
\label{figLeft}
\hfill
\includegraphics[width=7cm]{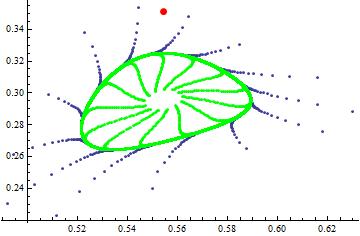}
\hfill
\caption{Fig.\ref{f131} and Fig.\ref{f132} in one system of coordinates}\label{c4}
\label{figRight}
\end{multicols}
\end{figure}

5) For $a=3.9, b=1, c=1, d=3.1, \alpha=1, x^{(0)}=0.1, y^{(0)}=0.01$ the trajectory is given in Fig.\ref{171}

6) For $a=4.4, b=1.3, c=1.1, d=3.1, \alpha=1, x^{(0)}=0.1, y^{(0)}=0.01$ the trajectory is given in Fig.\ref{172}

7) For $a=4.3, b=1, c=1, d=3.1, \alpha=1, x^{(0)}=0.1, y^{(0)}=0.01$ the trajectory is given in Fig.\ref{173}

8) For $a=4.4, b=1, c=1, d=3.1, \alpha=1, x^{(0)}=0.1, y^{(0)}=0.01$ the trajectory is given in Fig.\ref{174}

 \begin{figure}[h!]
\begin{multicols}{2}
\hfill
\includegraphics[width=7cm]{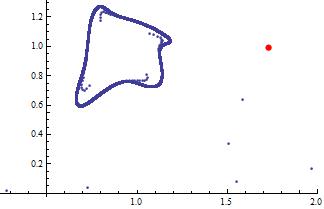}
\hfill
\caption{$a=3.9, b=1, c=1, d=3.1, \alpha=1, x^{(0)}=0.1, y^{(0)}=0.01$}\label{171}
\label{figLeft}
\hfill
\includegraphics[width=7cm]{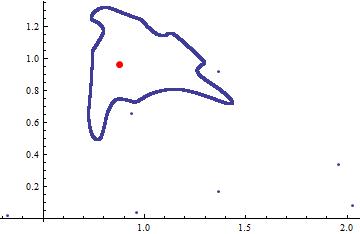}
\hfill
\caption{$a=4.4, b=1.3, c=1.1, d=3.1, \alpha=1, x^{(0)}=0.1, y^{(0)}=0.01$}\label{172}
\label{figRight}
\end{multicols}
\end{figure}

 \begin{figure}[h!]
\begin{multicols}{2}
\hfill
\includegraphics[width=7cm]{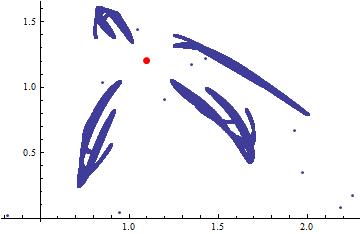}
\hfill
\caption{$a=4.3, b=1, c=1, d=3.1, \alpha=1, x^{(0)}=0.1, y^{(0)}=0.01$}\label{173}
\label{figLeft}
\hfill
\includegraphics[width=7cm]{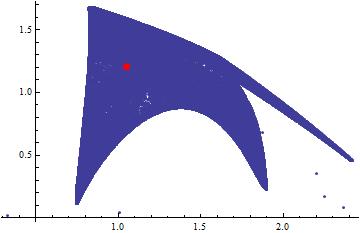}
\hfill
\caption{$a=4.4, b=1, c=1, d=3.1, \alpha=1, x^{(0)}=0.1, y^{(0)}=0.01$}\label{174}
\label{figRight}
\end{multicols}
\end{figure}

 \subsection{Lyapunov exponents}

It is known that the Lyapunov
exponents  describe  the  behavior  of  vectors  in  the tangent space of the phase space and are defined from the Jacobian matrix, (\cite{Arar}, \cite{Arro}, \cite{Luca}, \cite{Wu}).

 Lyapunov exponent  is calculated by eigen values of the  limit  of  the  following  expression:
$$(J_0J_1...J_n)^{\frac{1}{n}}$$
where $n$ tends to infinity, and $J_i$ is the Jacobian of the function at the iterated point $(x_i,y_i)$.  For  the  evaluation  of  Lyapunov  exponent,  we have taken an initial point and iterated it say for 106 time so that we are  closer to the fixed  point.   We  find   $J_0J_1...J_n$  where  $n=106$  say   and  calculate  the  Eigenvalues  of  that resultant matrix. Then $\lambda=\frac{\ln(eigenvalue)}{n}$ is the  Lyapunov exponent.

For the system (\ref{discr}) we consider the case $a=3.9, b=2, c=2, d=3.6, \alpha=3, x^{(0)}=0.5, y^{(0)}=0.4 $ and we calculate the Lyapunov exponent. The Jacobian is:

$$
J(x,y)=\begin{bmatrix}
2.9-4x-2y & -2x\\
3\frac{y^2}{x^2} & 2.6-6\frac{y}{x}
\end{bmatrix}$$

\begin{figure}[h!]
\begin{multicols}{2}
\hfill
\includegraphics[width=7cm]{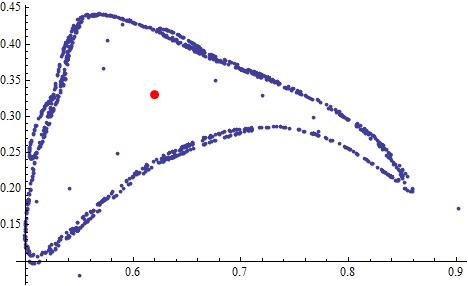}
\hfill
\caption{a=3.9, b=2, c=2, d=3.6, $\alpha=3, x^{(0)}=0.5, y^{(0)}=0.4.$  Shown $\lambda^{(n)}$,  $n=0,1,\dots,10000.$}\label{f14}
\label{figLeft}
\hfill
\includegraphics[width=7cm]{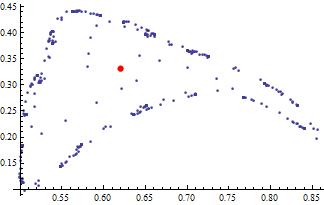}
\hfill
\caption{a=3.9, b=2, c=2, d=3.6, $\alpha=3, x^{(0)}=0.5857, y^{(0)}=0.319.$ Shown $\lambda^{(n)}$,  $n=0,1,\dots, 10000.$}\label{f14-1}
\label{figRight}
\end{multicols}
\end{figure}

 \begin{figure}[h!]
\begin{multicols}{2}
\hfill
\includegraphics[width=7cm]{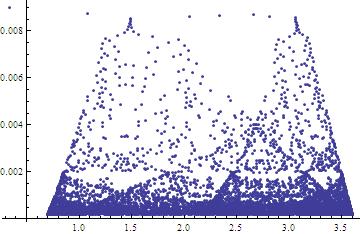}
\hfill
\caption{a=4.8, b=1, c=30, d=2, $\alpha=1, x^{(0)}=0.1, y^{(0)}=0.01.$ Shown $\lambda^{(n)}$,  $n=0,1,\dots, 10000.$}\label{181}
\label{figLeft}
\hfill
\includegraphics[width=7cm]{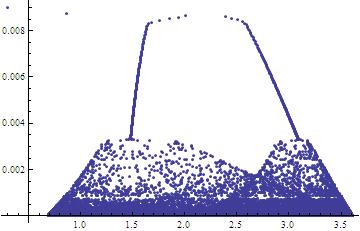}
\hfill
\caption{a=4.8, b=1, c=70, d=2, $\alpha=1, x^{(0)}=0.1, y^{(0)}=0.01.$ Shown $\lambda^{(n)}$,  $n=0,1,\dots, 100000.$}\label{182}
\label{figRight}
\end{multicols}
\end{figure}

 \begin{figure}[h!]
\begin{multicols}{2}
\hfill
\includegraphics[width=7cm]{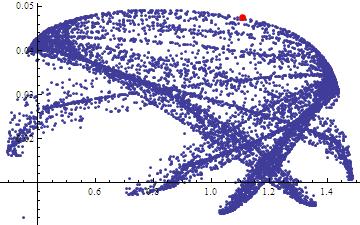}
\hfill
\caption{a=4.8, b=1.9, c=3, d=2.6, $\alpha=14, x^{(0)}=0.1, y^{(0)}=0.01.$  Shown $\lambda^{(n)}$,  $n=0,1,\dots, 10000.$}\label{183}
\label{figLeft}
\hfill
\includegraphics[width=7cm]{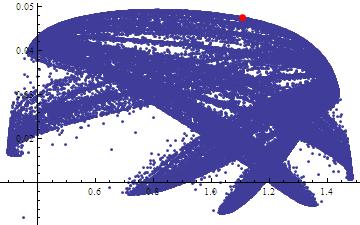}
\hfill
\caption{a=4.8, b=1.9, c=3, d=2.6, $\alpha=14, x^{(0)}=0.1, y^{(0)}=0.01.$  Shown $\lambda^{(n)}$,  $n=0,1,\dots, 100000.$}\label{184}
\label{figRight}
\end{multicols}
\end{figure}

Then
 $$J=J_0J_1...J_{106}=\begin{bmatrix}
-747074 & 278832\\
-3.41787*10^6 & 1.27566*10^6
\end{bmatrix}$$

 The eigenvalues of this matrix are $\mu_1=528588,\, \mu_2=-8.14907*10^{-10}$ and from this  the positive Lyapunov exponent is $$\lambda=\frac{\ln(528588)}{106}\approx 0.1243.$$
 Hence, in this case the trajectory is chaos \cite{Alli}.
 The dots in the XY-plane are given in Fig.\ref{f14}.

 \section*{Acknowledgements}

Shoyimardonov thanks the "El-Yurt Umidi" Foundation under the Cabinet of Ministers of the
Republic of Uzbekistan for financial support during his visit to the University of Montpellier (France) and prof. R.Varro for the invitation.


\begin{thebibliography}{1}

\bibitem{Alli}
K. Alligood, T. Sauer, and J. Yorke.
\newblock An Introduction to Dynamical Systems.
\newblock In {\em New York: Spinger-Verlag}, 1997.

\bibitem{Arar}
H.D.I Arardonel, R. Brown and M.B. Kennel.
\newblock Local   Lyapunov  Exponents     Computed  from Observed Data.
\newblock In {\em J. Nonlinear Science}, pages 175--199. 1991.

\bibitem{Arro}
D.K.  Arrowsmith    and  C.M.  Place.
\newblock An Introduction  to  Dynamical  Systems.
\newblock In {\em Cambridge Unversity Press}, 1994.

\bibitem{Aziz}
M.A. Aziz-Alaoui.
 \newblock Study of a Leslie-Gower-type tritrophic population model.
 \newblock In {\em Chaos, Solitons and Fractals}, pages 1275--1293. 2002.

\bibitem{Brit}
N. Britton.
\newblock Essential Mathematical Biology.
\newblock In {\em Springer, London}, 2003.

\bibitem{Ch}
Y. Chow, S.R.-J. Jang.
\newblock Asymptotic dynamics of a modified discrete Leslie-Gower competition system.
 \newblock In {\em Int. J. Biomath}, 23 pp. 2017.

\bibitem{De}
R.L. Devaney.
\newblock  An Introduction to Chaotic Dynamical System.
\newblock In {\em Westview Press}, 2003.

\bibitem{Luca}
L. Diect, R. D. Russell, E. S. Van Vleck.
\newblock  On the Computation  of Lyapunov Exponents  for  Continuous Dynamical  Systems.
\newblock In {\em SIAM Journal, Numer. Anal}, pages 402--423, \textbf{34}(1) (1997).

\bibitem{RGan}
R.N. Ganikhodzhaev.
 \newblock Quadratic stochastic operators, Lyapunov functions and tournaments.
 \newblock In {\em Russian Acad. Sci.Sb. Math}, pages 489--506. 76 (1993).

\bibitem{GMR}
R.N. Ganikhodzhaev, F.M. Mukhamedov, U.A. Rozikov.
 \newblock  Quadratic stochastic operators and processes: results and open problems.
  \newblock In {\em Inf. Dim. Anal. Quant. Prob. Rel. Fields}, pages 279--335, 14(2) (2011).


\bibitem{GU2}
R.P. Gupta, P. Chandra,  M. Banerjee.
\newblock Dynamical complexity of a preypredator model with nonlinear predator harvesting.
 \newblock In {\em Discr. Cont. Dyn. Sys.}, Series B,pages 423--443, \textbf{20}(2) (2015).

\bibitem{Rob}
R. C.  Hilborn.
\newblock Chaos  and Nonlinear Dynamics,  An  Introduction For  Scientists  and Engineers.
\newblock In {\em Oxford University Press}, 1994.


\bibitem{M}
J. M\"uller, C. Kuttler.
\newblock Methods and models in mathematical biology.
 \newblock In {\em Deterministic and stochastic approaches. Lecture Notes on Mathematical Modelling in the Life Sciences}. Springer, Heidelberg, 2015.

\bibitem{RSH}  U.A. Rozikov, S.K. Shoyimardonov.
\newblock Ocean ecosystem discrete time dynamics generated by $l-$ Volterra operators.
 \newblock In {\em Inter. Jour. Biomath}. \textbf{12}(2) (2019) 24 pages.

\bibitem{RU}
U.A. Rozikov, J.B. Usmonov.
\newblock Dynamics of a population with two equal dominated species.
\newblock In {\em 	arXiv:1909.07106 [math.DS]}, 2019.


\bibitem{RV}
U.A. Rozikov, M.V. Velasco.
\newblock  A discrete-time dynamical system and an evolution algebra of mosquito population.
 \newblock In {\em J. Math. Biol}. pages 1225--1244, \textbf{78}(4) (2019).

\bibitem{Shar}
A.N. Sharkovskii, S.F. Kolyada, A.G. Sivak, V.V. Fedorenko.
\newblock  Dynamics of one-dimentional maps.
\newblock In {\em Kiev, Naukova Dumka}, 1989.

\bibitem{SB}
R. Sivasamy, K. Sathiyanathan, K. Balachandran.
\newblock Dynamics of a modified Leslie-Gower model with gestation effect and nonlinear harvesting.
 \newblock In {\em J. Appl. Anal. Comput}, pages 747--764,  \textbf{9}(2) (2019).

\bibitem{SS}
S. Slimani, P.F. Raynaud, I. Boussaada.
\newblock Dynamics of a prey-predator system with modified Leslie-Gower and Holling type II schemes incorporating a prey refuge.
\newblock In {\em  Discrete Contin. Dyn. Syst. Ser. B}, pages 5003--5039, \textbf{24}(9) (2019).

\bibitem{Wu}
H.M. Wu.
\newblock The Hausdroff dimension of chaotic sets generated by a continuous map from $[a,b]$ into itself.
  \newblock In {\em J. South China Univ. Natur. Sci. Ed.}, pages  45--51, 2002.
  
 \end{thebibliography}

\end{document}